\documentclass[11pt, one side]{article}
\usepackage[english]{babel}
\usepackage[centertags]{amsmath}
\usepackage{amssymb}
\usepackage{amsthm}
\usepackage{newlfont}
\usepackage{amsfonts}
\usepackage{bigints}
\usepackage{hyperref}
\usepackage{hypcap}

\theoremstyle{plain}

\newtheorem{theorem}{Theorem}[section]

\newtheorem{definition}{Definition}[section]

\newtheorem{lemma}{Lemma}[section]

\newtheorem{proposition}{Proposition}[section]
\newtheorem{remark}{Remark}[section]

\newcommand{\pref}[1]{Proposition~\textup{\ref{#1}}}
\newcommand{\cref}[1]{Corollary~\textup{\ref{#1}}}

\begin{document}

\begin{center}
\textbf{\LARGE Variational problems on product spaces \\\vspace{5pt} Different obstacle constraints }
\end{center}

 \begin{center}
\emph{Antoine Mhanna}
\end{center}
  \begin{center} \emph{1 Dept of Mathematics, Lebanese University, Hadath, Beirut, Lebanon.
}\end{center}
\begin{center}
\emph{tmhanat@yahoo.com}\end{center}
\centerline{\textbf{ Abstract}}\vspace{08pt}
We study two principle minimizing problems,  subject of different constraints. 
Our open sets are assumed bounded, except  mentioning otherwise;
precisely $\Omega=]0,1[^n \in {\mathbb{R}}^n , n=1 $ or $n=2$.
 \section{Introduction}
\begin{theorem}[ Rellich-Kondrachov theorem]
Suppose $\Omega$ is bounded and of class $C^1$ then, 
$W^{1,p} \subset  L^p $
with compact injection for all $p$ (and all $n$).
\end{theorem}

\noindent Let $p \ge 2 $ and $W^{1,p}(]0,1[;\mathbb{R}^2)=\{u=(u_1,u_2); 
u_1 \in 
W^{1,p}(]0,1[;\mathbb{R}) ,u_2\in W^{1,p}(]0,1[;\mathbb{R})\}$

Define the functionals
\begin{enumerate}

\item$\displaystyle  W^{1,p}(]0,1[;\mathbb{R}^2)\to{ \mathbb{R}}_+  \\\text{\hspace{65pt}}  
u \to  F(u)=\int_0^1{|u'(x)|}_\text{\tiny{2}}^p=\int_0^1\left({ |u_1'(x)|}^2+{|u_2'(x)|}^2\right)^\frac{p}{2}$

 \item $\displaystyle W_0^{1,p}({\Omega};\mathbb{R}^2)\to{ \mathbb{R}}_+ \\\text{\hspace{51pt}} u\to K(u)=\int_{\Omega}{|\nabla u(x)|}_\text{\tiny{2}}^p=\int_{\Omega}\left({ |\nabla u_1(x)|}_\text{\tiny{2}}^2+{|\nabla u_2(x)|}_\text{\tiny{2}}^2\right)^\frac{p}{2}$
\end{enumerate}
 \indent Mainly, our goals are:\begin{itemize}
\item  show if that there exists $ u_0  \in A_i $ unique such that, $G(u_0)=\inf\{G(u); u\in A_i\}$
\item  write the Euler-Lagrange equation satisfied by a  'smooth' $u_0$
\end{itemize}

Let us define the constraint sets:\begin{enumerate}

\item $A_1=\{u\in W^{1,p}(]0,1[;\mathbb{R}^2):
\\ |u|_{\text{\tiny{2}}}^2=\left(|u_1|^2+ |u_2|^2\right)=1\textbf{  a.e.}  \text{ so    } |u_i|_{\infty}\le 1    ,u_1^2=1-u_2^2; u(0)=(0,1), u(1)=(1,0)\}$

\item$A_2=\{u\in W^{1,p}(\Omega;\mathbb{R}^2)  ; u_1=0   \text{ \&   } u_2=1   \text  {  on  }   \partial \Omega ;u_1\in W_0^{1,p}(\Omega),
|u|_{\text{\tiny{2}}}^2=\left(|u_1|^2+ |u_2|^2\right)=1\textbf{  a.e.} \text{ so    } |u_i|_{\infty}\le 1  ,u_1^2=1-u_2^2 \}$

 \end{enumerate}
Note that the condition \textbf{  a.e.} is implicitly important. One can notice that it could be written directly  into  equation $u_2=\sqrt{1-(u_1)^2};$ without loss of generality we didn't do so. Clearly, boundary condition does not define a   vector space,  if $u_1(0)=0, u_1(1)=1 ,$ we write $u_1=g$ and $u_2=1-g$ 
on $\partial \Omega$ and  $g$ may be a function defined on the open set $\Omega$ as well. 
\section{Solutions }
\begin{lemma}
$A_i \neq \phi$ for all  $ i$.
\end{lemma}
\begin{proof}
For $i=1$  consider the  bounded smooth functionals\\
 $ x \to u_1=\begin{cases}\exp{\left(\dfrac{x}{x^p-1}\right)}  \hspace{3pt} $  \text{ for any p} $> 0$\text{ if}  $\hspace{3pt}   x \in [0.1[ \\ 0 \hspace{77pt}  \text{if not }   \end{cases}$\\

For $i=2$, similarly  but more explicitly we use the following proposition about partition of unity which lead to the result after a regularization process.
\end{proof}
\begin{proposition}
let $\Omega $ be an open set of ${\mathbb{R}}^d$ and $K$ a compact   $   \subset \Omega $.\\
Then  $\exists  $  $ \Phi $  $\in C_c ({\mathbb{R}}^d)$, such that
$$0\le \Phi \le 1 ,\hspace{15pt} supp(\Phi) \subset \Omega.$$
\end{proposition}

\begin{definition}
The $p$-norm on ${\mathbb{R}}^n$  is defined as:
$$x=(x_1,\cdots,x_n)\in {\mathbb{R}}^n, p \in ]0;+\infty[: x\to |x|_p=\left(\sum_{i=1}^n|x_i|^p\right)^{\frac{1}{p}}$$ and  it is denoted by $|.|_p.$
\end{definition}
\begin{lemma}
If $ u_1 \in L^p(\Omega), $ and $u_2 \in L^p(\Omega) $
then $ \left({|u_1(x)|}^2+{|u_2(x)|}^2\right)^\frac{p}{2} \in L^1(\Omega)$
\end{lemma}
\begin{proof}
Write $|u|_{\text{\tiny{$2$}}}     \le  C|u|_{p}$
for some $C>0$.
\end{proof}

\subsection{Existence and uniqueness}

Note that product spaces such $V\times V $  are equipped with the sum norm that is $ \|u\|+\|v\|$.\\
Usually we  will study $K(u)^\frac{1}{p}, $ and $  F(u)^\frac{1}{p}$ as the norm $L^p$ will appear explicitly. Before we state the main theorem, we have:\\
\begin{proposition}\label{gh}
$|v(\partial\Omega)|\le C\|v\|_{W^{1,p}}\text{\hspace{11pt}}   \forall v \in {W^{1,p}} ]0,1[ \\\text{where }  \partial\Omega:=\{0,1\}$
\end{proposition}
\begin{remark}
A minimizer of a positive valued function $f$ is also a minimizer of $f^p$ and conversely , $\forall p>0$.
\end{remark}
\begin{theorem}\hfill
\begin{enumerate}
\item There exists at least one function   $ u=(u_1,u_2) \in A_1  $  solving $\displaystyle F(u)=\underset{w\in A_1}{\min}\text{\hspace{4pt}}   F(w).$
\item There exists at least one function   $ u=(u_1,u_2) \in A_2   $    solving $\displaystyle K(u)=\underset{w\in A_2}{\min}\text{\hspace{4pt}}   K(w).$
\end{enumerate}
\end{theorem}
\begin{proof}
First $ F(u)^\frac{1}{p}$  and   $K(u)^\frac{1}{p}$ are   both continuous convex functions thus weakly lower semi continuous.
Also the constraints sets are weakly closed, in the sense that,  if $u_n \rightharpoonup u ,$  and $u_n$ satisfies any of the constraint, $u$ will be as well .
For the boundary condition that is $u=g$ on the boundary, choose any $h$ satisfying same constraints,  $u_n-h$  is a sequence  $ \in W_0^{1,p}\times W_0^{1,p}$,   a convex closed subspace of
$ W^{1,p}\times W^{1,p}$, hence weakly closed.\\
For the condition of $|u_i|_{\infty} \le 1$   \hspace{4pt}a.e. \hspace{5pt }it suffices to show that $|u_1|_{\infty} \le 1$ \hspace{4pt} a.e. Take a sequence weakly convergent  to $u $  in  $ W^{1,p}$ by Rellich-K. Theorem we have a strong convergence at least in one of  the $L^p$'s. Thus we can extract a subsequence that converges a.e. to $u$. Giving that $|\Omega| < \infty$, by Egoroff theorem the a.e convergence is equivalent to uniform convergence, up to arbitrarily negligible sets. Since the set is closed for the uniform convergence, we conclude that  $|u_i|_{\infty} \le 1$  ,\hspace{4pt}$i=1,2$ \textbf{ a.e}\\

It could  be said  directly  after the extraction of a subsequence a.e. convergent, that we have $$
 |{u_{k_j}}_1|^2+| {u_{k_j}}_2|^2 \to |u_1|^2+|u_2|^2=1  \textbf{ a.e}$$
Remaining  to show that the functionals verify a coercivity condition over the product space.\\
\begin {enumerate}
\item Set $ f:=\inf _{u\in A_1} F(u).$ If $f=+\infty$ we are done, suppose  $f$   is finite. Select
a minimizing sequence$\{u_k\}$, then $F(u_k)\to f $ because we are in $\mathbb{R}$
 $$F(u)\ge  C.\left(\|u_1\|_{W_0^{1,p}}^p+\|u_2\|_{W_0^{1,p}}^p\right)\ge CC'\left(\|u_1\|_{W_0^{1,p}}+\|u_2\|_{W_0^{1,p}}\right)^p$$
One can verify because of boundary conditions (on $A_1$) that we have equivalence between  the two norms $\|.\|_{W_0^{1,p}}$ and  $\|.\|_{W^{1,p}}$
i.e.$$F(u_k)\ge \alpha \|u_k\|:=\alpha \left(\|{u_{k_1}}\| _{W^{1,p}} +\|{u_{k_2}}\|_{W^{1,p}}\right)^p.$$ This estimate  implies that $\{u_k\} $ is bounded in ${W^{1,p}}\times{W^{1,p}}$. Consequently there exist a subsequence $\{{u_{k_j}}\}$ and a function   $u\in {W^{1,p}}\times{W^{1,p}}$ such that; ${u_{k_j }} \rightharpoonup u $   weakly in ${W^{1,p}}\times{W^{1,p}}$,
thus $ F(u)$ is weakly lower semi continuous.
$F(u)\le \lim $  $ \inf_{j\to\infty} F({u_{k_j}})=f$, since $u\in A_1$ it follows that $$F(u)=f=\underset{u\in A_1}{min}\text{\hspace{4pt}}   F(u).$$
\item Similarly, set $ m:=\inf _{u\in A_2} K(u).$ If $m=+\infty$ we are done, suppose   $m$  is finite, select a minimizing sequence $\{u_k\}$, then $K(u_k)\to m$ because we are in $\mathbb{R}$.
{\begin{align} \nonumber  K(u)=&\int_{\Omega}{|\nabla u(x)|}_\text{\tiny{2}}^p=\int_{\Omega}\left({ |\nabla u_1(x)|}_\text{\tiny{2}}^2+{|\nabla u_2(x)|}_\text{\tiny{2}}^2\right)^\frac{p}{2}\\ \nonumber \nonumber
K(u)=&\int_{\Omega} \left((\frac{\partial u_1}{\partial x_1})^2 +(\frac{\partial u_1}{\partial x_2})^2 +(\frac{\partial u_2}{\partial x_1})^2+(\frac{\partial u_2}{\partial x_2})^2\right)^\frac{p}{2} \nonumber \\[0.28cm] \nonumber
 \ge&\text{      }   C\int_{\Omega}(\frac{\partial u_1}{\partial x_1})^p +(\frac{\partial u_1}{\partial x_2})^p +(\frac{\partial u_2}{\partial x_1})^p+(\frac{\partial u_2}{\partial x_2})^p\\[0.28cm]
\ge& \text{      }  CC'(\|u_{1}\|_{W^{1,p}}^p) +C(\|\nabla \sqrt{1-{u_{1}}^2}\|_{L^p}^p)\end{align}}
Consequently \normalsize\begin{equation}K(u_k) \ge min( CC',C) (\|u_{k_1}\|_{W^{1,p}}^p +\|\nabla \sqrt{1-{u_{k_1}}^2}\|_{L^p}^p),\end{equation}
and $u_1$ is bounded.  But if $u_1$ is bounded  so is  $u_2$ and conversely for:  $$1-\|u_1\|^2\le |1-\|{u_1}^2\|| \le \|1-{u_1}^2\| = \|{u_2}^2\|\le \|u_2\|^2$$ 
Thus we conclude  that  the sequence $\{u_k\}$ is bounded in $ W^{1,p} \times W^{1,p}$ and the proof is similar to  that of $F(u)$.

\end{enumerate}
\end{proof}

\begin{theorem}
The minimizing problem: $F(u)=\underset{w\in A_1}{\min}\text{\hspace{4pt}}   F(w)$  has a unique solution

\end{theorem}
\begin{proof}
Suppose not, if $u$ is a minimizer and a   distinct minimizer  $v$ exists, $v:=(v_1,v_2)$   write  $w= (w_1,w_2)=\left(\dfrac{\text{\small{$u_1+v_1$}}}{\text{\small{2}}},\dfrac{\text{\small{$u_2+v_2$}}}{\text{\small{2}}}\right)$
and  recall that the Euclidean norm $|.|_\text{\tiny{2}}$ is  strictly convex,  which  means  that as long as
\begin{equation} v' \neq \alpha.u'   \text{ a.e}.
\end{equation}   \\ 
we have this strict  inequality:\begin{eqnarray}{G(w)}^{\frac{1}{p}}={\left[\int_0^1{|w'(x)|}_{\text{\tiny{2}}}^p\right]}^{\frac{1}{p}}&=&{\left[\int_0^1{\left(\left(\frac{u'_1+v'_1}{2}\right)^2+\left(\frac{u'_2+v'_2}{2}\right)^2\right)}^{\frac{p}{2}}\right]}^{\frac{1}{p}}\\
&<&\dfrac{1}{2}{\left[\int_0^1{\left(|u'|_{\text{\tiny{2}}}+|v'|_{\text{\tiny{2}}}\right)}^{p}\right]}^{\frac{1}{p}} \\&\le& \frac{1}{2}{ G(u) }^{\frac{1}{p}} + \frac{1}{2}{ G(v) }^{\frac{1}{p}} \end{eqnarray}\\ which contradicts the minimum property.
This contradiction completes the proof if we showed  that  $v' \neq \alpha.u'   \text{ a.e},$ suppose the converse and let $u=\beta v +cte,$  if  $u_1= {\beta}_1 v_1 + {cte}_1,$
applying boundary constraints and using \pref{gh} we conclude that $u_1 \neq  {\beta}_1 v_1 + {cte}_1$ a.e. 
 for any ${\beta}_1$
and any constant ${cte_1}$
\end{proof}
\section {Euler-Lagrange}
\begin{lemma}
$F(u) $ and  $K(u)$ are both differentiable $({C}^1)$ on the product space  except  at $(0,0)$
\end{lemma}
\vspace{3mm}
\begin{proof} This follows by the regularity of the ${|.|}_\text{\tiny{2}}$ norm and derivation under integral sign.
\end{proof}
From this, we can compute  the Euler-Lagrange equations giving the existence of a minimizer $({u_{0_1}},{u_{0_2}})\neq 0 $. Bearing in mind that   $C^1$  Gateaux differentiable is the same as   ${C}^1$ Frechet -differentiable. We will give the  'equation' satisfied by the 'minimizer' of $K(u)$ as it 
is  the most  general  case.\\
Fix $v\in W_0^{1,p}(\Omega,{\mathbb{R}}^2)\cap L^{\infty}(\Omega,{\mathbb{R}}^2)$. Since ${|u|}_\text{\tiny{2}} =1  \text{  a.e} $, we have $$|u+\tau v|_\text{\tiny{2}}  \neq 0  \text{  a.e}.$$
for each sufficiently small $\tau$ by continuity. Consequently 
$$ v(\tau):=\dfrac{u+\tau v }{|u+\tau v|_{\text{\tiny{2}}}} \in A_2 $$
Thus $$ k(\tau):=K(v(\tau)$$
has a minimum at   $ \tau=0,$ and so $$k'(0)=0.$$ Norms on product spaces  are of course Euclidien norms, that is $|.|_\text{\tiny{2}}.$ Matrices such the gradient matrix   (here it's a  $2\times 2$ matrix)  can be identified to a vector  $\in {\mathbb{R}}^4$,
and let $( . )$ denotes the usual scalar product on $ {\mathbb{R}}^n$, by a direct computation we have:
\begin{proposition}
$v'(\tau)= \dfrac{v}{|u+\tau v|} - \dfrac{[(u+\tau v) .v](u+\tau v)}{|u+\tau v|^3}$
\end{proposition}

\begin{theorem}
Let $u\in A_2$ satisfy\\
$$K(u)=\underset{w\in A_2}{min}\text{\hspace{4pt}}   K(w).$$
Then 
\begin{equation}
\int_{\Omega} p|Du|^{\frac{p}{2}-1} [(Du.Dv)-|Du|^2(u.v)]
\end{equation}
for each $v \in W_0^{1,p}(\Omega,{\mathbb{R}}^2) \cap L^{\infty}(\Omega,{\mathbb{R}}^2).$
\end{theorem}
\begin{proof}
In fact $$K(u)= \int_{\Omega} |Du|^p $$
where $Du$ is the gradient matrix associated to $u$ and the norm as said  is the one associated to the scalar product:$<A,B>= Tr(B^t.A)$.
 \begin{equation}k'(0)=0 = \int_\Omega p|Du|^{\frac{p}{2}-1} \text { }  Du.Dv'(0) \end{equation}$$=\int_\Omega p|Du|^{\frac{p}{2}-1} Du.D(v-(u.v)u)$$
Upon differentiating $|u|^2=1$ a.e., we have $$(Du)^T u=0$$\\
Using this fact, we then verify $$Du.(D(u.v)u)=|Du|^2(u.v)      \text{ \hspace{4pt} a.e. in $\Omega$}$$
This identity employed in (7) gives (6). We leave details to the interested reader.
\cite{key2}

\end{proof}

\end{document}